\documentclass{amsart}
\usepackage[utf8]{inputenc}
\usepackage[english]{babel}
\usepackage{csquotes}
\usepackage{cite}

\usepackage{amsmath}
\usepackage{amssymb}
\usepackage{amsthm}

\usepackage{enumerate}
\usepackage{xcolor}
\usepackage{graphicx}
\usepackage{caption}
\usepackage{subcaption}
\usepackage{textcomp}
\usepackage{fancyhdr}
\usepackage[normalem]{ulem}

\allowdisplaybreaks

\newtheorem{theorem}{Theorem}[section]
\newtheorem*{theorem*}{Theorem}
\newtheorem{corollary}[theorem]{Corollary}
\newtheorem{lemma}[theorem]{Lemma}
\newtheorem{prop}[theorem]{Proposition}
\newtheorem*{prop*}{Proposition}

\newtheorem*{embedding-corollary}{Corollary~\ref{Cor:Embedding}}
\newtheorem*{main-result}{Theorem~\ref{Theorem:Main-Result}}

\theoremstyle{definition}
\newtheorem{definition}[theorem]{Definition}

\theoremstyle{remark}
\newtheorem{remark}[theorem]{Remark}

\newcommand{\A}{\mathcal A}

\newcommand{\N}{\mathbb N}

\DeclareMathOperator{\opint}{int}
\renewcommand{\int}{\opint}

\DeclareMathOperator{\dist}{dist}

\definecolor{darkgreen}{rgb}{0,.66,0}
\definecolor{lightgreen}{rgb}{.5,.83,.5}

\title{
    Homomorphisms from aperiodic subshifts to subshifts with the finite extension property
}

\author{Robert Bland}
\address{Robert Bland\\
Department of Mathematics and Statistics\\
University of North Carolina at Charlotte \\
9201 University City Blvd.\\
Charlotte, NC 28223}
\email{rbland5@charlotte.edu}

\author{Kevin McGoff}
\address{Kevin McGoff\\
Department of Mathematics and Statistics\\
University of North Carolina at Charlotte \\
9201 University City Blvd.\\
Charlotte, NC 28223}
\email{kmcgoff1@charlotte.edu}
\urladdr{https://pages.charlotte.edu/kevin-mcgoff/}

\begin{document}

\begin{abstract}
Given a countable group $G$ and two subshifts $X$ and $Y$ over $G$, a continuous, shift-commuting map $\phi : X \to Y$ is called a homomorphism. Our main result states that if every finitely generated subgroup of $G$ has polynomial growth, $X$ is aperiodic, and $Y$ has the finite extension property (FEP), then there exists a homomorphism $\phi : X \to Y$. 
By combining this theorem with the main result of \cite{bland}, we obtain that if the same conditions hold, and if additionally the topological entropy of $X$ is less than the topological entropy of $Y$ and $Y$ has no global period, then $X$ embeds into $Y$. We also establish some facts about subshifts with the FEP that may be of independent interest.
\end{abstract}

\maketitle

\section{Introduction}

Consider a countable group $G$ and two subshifts $X$ and $Y$ over $G$. We are interested in conditions that guarantee the existence of a continuous, shift-commuting map $\phi : X \to Y$, known as a homomorphism. The study of subshifts via the homomorphisms between them has played a central role in symbolic dynamics (see \cite{lind_marcus} for an introduction). 
Indeed, when $G$ is the group of integers, the coding results known as Krieger's Embedding Theorem \cite[Theorem~3]{krieger} and Boyle's Lower Entropy Factor Theorem \cite[Theorem~2.5]{boyle}  serve as cornerstone results in the study of subshifts of finite type (SFTs).

Several more recent results have developed the coding theory of SFTs over increasingly general classes of groups (e.g., see \cite{bland,briceno_mcgoff_pavlov,huczek_kopacz,lightwood_1,lightwood_2,meyerovitch}), and the present work can be viewed as continuing in this direction. More specifically, we are largely motivated by recent efforts to generalize Krieger's Embedding Theorem to more general groups. The first such result, due to Lightwood \cite{lightwood_1}, holds for $G = \mathbb{Z}^d$. For notation, we let $h(X)$ denote the topological entropy of a subshift $X$.
\begin{theorem}[Lightwood \cite{lightwood_1}] \label{Thm:Lightwood1}
    If $X$ is an aperiodic $\mathbb{Z}^d$ subshift, $Y$ is a $\mathbb{Z}^d$ square-mixing SFT which contains a finite orbit, and there exists a homomorphism $\phi : X \to Y$, then there exists an embedding of $X$ into $Y$ if and only if $h(X) < h(Y)$.
\end{theorem}
Observe that in addition to the structural assumptions on $X$ (aperiodicity) and $Y$ (square-mixing SFT with a finite orbit), this result contains the hypothesis that there exists a homomorphism $\phi : X \to Y$. In an effort to find sufficient conditions to guarantee that this hypothesis holds, Lightwood went on to establish the following complementary theorem for $G = \mathbb{Z}^2$.
\begin{theorem}[Lightwood \cite{lightwood_2}] \label{Thm:LightwoodHom}
If $X$ is an aperiodic $\mathbb{Z}^2$ subshift and $Y$ is a $\mathbb{Z}^2$ square-filling mixing SFT, then there exists a homomorphism $\phi : X \to Y$.
\end{theorem}
Combining these two theorems, Lightwood obtained a $\mathbb{Z}^2$ generalization of Krieger's result in the setting where $X$ is aperiodic and $Y$ is square-filling mixing. 
We are not aware of any other results that address the case $G = \mathbb{Z}^d$ for $d \geq 3$.

{In more recent work}, the first author established the following theorem for all countable amenable groups with the comparison property \cite[Definition~6.1]{downarowicz_zhang-symb-ext}, an algebraic property based on the idea of dynamical comparison \cite[Definition~3.2]{kerr}. Many groups have been shown to have the comparison property, including all countable ``subexponential" groups (i.e., those not containing a finitely generated subgroup of exponential growth) \cite[Theorem~6.33]{downarowicz_zhang-symb-ext}. 

\begin{theorem}[Bland \cite{bland}] \label{Thm:bland}
    Let $G$ be a countable amenable group with the comparison property. Let $X$ be a nonempty aperiodic $G$ subshift, and let $Y$ be a strongly irreducible SFT with no global period. If $h(X) < h(Y)$ and there exists a homomorhpism $\phi : X \to Y$, then there exists an embedding of $X$ into $Y$.
\end{theorem}
Observe that this theorem also contains the hypothesis that there exists a homomorphism $\phi : X \to Y$. In an effort to establish sufficient conditions for this hypothesis to hold, we establish the following theorem, which is our main result. 
For definitions, see Section \ref{Sect:Background}.
\begin{theorem} \label{Theorem:Main-Result}
Let $G$ be a countable group such that every finitely generated subgroup of $G$ has polynomial growth. 
Let $X$ be an aperiodic $G$-subshift, and let $Y$ be a nonempty $G$-subshift with the finite extension property. Then there exists a homomorphism $\phi : X \to Y$. 
\end{theorem}

The proof of this result appears in Section \ref{Sect:Proof}. We remark that this theorem provides conditions for the existence of homomorphisms when $G = \mathbb{Z}^d$, for all $d \in \N$, and in this sense it can be viewed as a generalization of Lightwood's result for $\mathbb{Z}^2$ (Theorem \ref{Thm:LightwoodHom} above). We remark that although the finite extension property (FEP) is similar in spirit to Lightwood's square-filling mixing property, the FEP is a stronger condition, and therefore Theorem \ref{Theorem:Main-Result} is not a proper generalization of Lightwood's result.

By combining Theorem \ref{Theorem:Main-Result} with Theorem \ref{Thm:bland}, we obtain the following embedding theorem.
\begin{corollary} \label{Cor:Embedding} Let $G$ be a countable group such that every finitely generated subgroup of $G$ has polynomial growth. 
Suppose $X$ is a nonempty aperiodic $G$-subshift and $Y$ is a $G$-subshift with the finite extension property and no global period. If $h(X) < h(Y)$, then $X$ embeds into $Y$.
\end{corollary}

The finite extension property has been defined previously for $\mathbb{Z}^d$ subshifts \cite{briceno_mcgoff_pavlov}, and it may be viewed as a particularly strong condition. Indeed, if $X$ has the FEP, then $X$ must be a strongly irreducible SFT (Proposition \ref{Prop:SIandSFT}). Nonetheless, there are many previously considered systems that have the FEP, including full shifts, SFTs with a safe symbol, and  SFTs with \textit{single-site fillability} (such as ``checkerboard" systems with sufficiently many colors). See \cite{briceno_mcgoff_pavlov,briceno,marcus_pavlov} for additional examples and discussion. In Section \ref{Sect:FEP} we also show that if $G$ is a countable amenable group and $X$ has the FEP, then $X$ must be entropy minimal (Proposition \ref{Prop:EntropyMinimality}). As a result, one may replace the last sentence of Corollary \ref{Cor:Embedding} with the following statement: $X$ is isomorphic to a strict subsystem of $Y$ if and only if $h(X) < h(Y)$.

In our proofs, we use that if a finitely generated group has polynomial growth, then it satisfies a property known as the doubling property (see Section \ref{Sect:CountableGroups} for details). Although the doubling property appears similar in spirit to the Tempelman condition \cite[Definition~1.1]{hochman}, we do not know how to make our proofs work with only the Tempelman condition in place of the doubling property.

Finally, we mention that Meyerovitch has recently established an embedding theorem for a class of subshifts of finite type over countable abelian groups \cite{meyerovitch}. In particular, \cite{meyerovitch} defines a property called the map extension property and then gives necessary and sufficient conditions for embedding an arbitrary $G$-subshift into a $G$-SFT with the map extension property. We note that by results of Poirier and Salo \cite{poirierandsalo}, the finite extension property is strictly more general than the map extension property (i.e., the map extension property implies the finite extension property, but the reverse is not true).

The paper is organized as follows. In Section \ref{Sect:Background} we provide all necessary definitions and background. Section \ref{Sect:FEP} contains our results concerning systems with the FEP. In Section \ref{Sect:Lemmas} we prove several preliminary results, 
in Section \ref{Sect:Proof} we present the proof of Theorem~\ref{Theorem:Main-Result}, and in Section \ref{Sect:Cor} we provide a short proof of Corollary~\ref{Cor:Embedding}.

\section{Background and notation} \label{Sect:Background}

\subsection{Countable groups} \label{Sect:CountableGroups}

Let $G$ be a countable group. If $K \subset G$ is a finite, symmetric set, then we let $\rho_K$ denote the standard left-invariant word metric on the subgroup generated by $K$, denoted by $\langle K \rangle$, and we let $B_n(K)$ denote the ball of radius $n$ centered at the identity with respect to $\rho_K$. Note that if $K$ contains the identity, then $B_n(K) = K^n$. We extend $\rho_K$ to all of $G$ by setting $\rho_K(g,h) = \rho_K(e,g^{-1}h)$ if $g^{-1}h \in$ {$\langle K\rangle$} and $\rho_K(g,h) = \infty$ otherwise. For any $g \in G$ and any nonempty finite set $S \subset G$, we also define
\begin{equation*}
\dist_K(g,S) = \min \{ \rho_K(g,s) : s \in S\}.
\end{equation*}

A group $G$ is finitely generated if there exists a finite subset $K$ such that $G = \langle K \rangle$, in which case we say that $K$ is a generating set. Furthermore, a finitely generated group $G$ is said to have \textit{polynomial growth} if there exists a finite generating set $K$ and constants $C,d > 0$ such that $|B_n(K)| \leq Cn^d$ for all $n \geq 1$. The infimum $d_0 \geq 0$ of all such $d$ for which this holds is called the \textit{order} of the polynomial growth of $G$. Though the growth function $f(n) = |B_n(K)|$ depends on the generating set $K$, the property of having polynomial growth does not depend on the specific choice of generating set.  

It is a classical theorem of Gromov that a finitely generated group $G$ has polynomial growth if and only if $G$ is virtually nilpotent \cite{gromov}. It follows that the order $d_0$ must be an integer and, moreover, there exists a $C > 0$ such that $n^{d_0}/C \leq |B_n(K)| \leq Cn^{d_0}$ for all $n \geq 1$.

\begin{definition} 
Suppose $G$ is finitely generated. We say that $G$ has the \textit{doubling property} if for any finite symmetric generating set $K \subset G$, there exists a constant $C$ depending on $K$ such that for any $n \geq 1$, we have
\begin{equation*}
|B_{2n}(K)|\leq C |B_n(K)|.
\end{equation*}
\end{definition}

\begin{remark}
It is easy to see that if $G$ has polynomial growth, then $G$ satisfies the doubling property. Indeed, if $n^d/C \leq |B_n(K)| \leq Cn^d$ for all $n$, then \[
    |B_{2n}(K)| \leq C(2n)^d = 2^dC^2(n^d/C) \leq 2^dC^2|B_n(K)|.
\] In fact, the converse is also true. If $G$ satisfies the doubling property, then $G$ must be of polynomial growth, by the following argument. 
Suppose $|B_{2n}(K)| \leq C |B_n(K)|$ for all $n \geq 1$, and suppose without loss of generality that $C \geq |B_1(K)|$. By induction, $|B_{2^n}(K)| \leq C^{n+1}$ for all $n \geq 0$. Now let $m \geq 1$ be arbitrary and set $n = \lfloor \log_2 m\rfloor$. We then have that \[
    |B_m(K)| \leq |B_{2^{n+1}}(K)| \leq C^{n+1} = C(2^n)^{\log_2 C} \leq Cm^{\log_2 C}.
\]
This demonstrates that $G$ must be of polynomial growth of order at most $\lfloor \log_2 C \rfloor$. We conclude that the doubling property is equivalent to polynomial growth. 
\end{remark}

We conclude this subsection with a final bit of notation. For any countable group $G$, if $K \subset G$ is finite and $E \subset G$, then we let
\begin{equation*}
\int_K(E) = \{ g \in E : gK \subset E\}.
\end{equation*}

\subsection{Symbolic dynamics}

Let $G$ be a countable group. For any finite nonempty set $\mathcal{A}$ endowed with the discrete topology, we endow $\Sigma(\mathcal{A}) = \mathcal{A}^G$ with the product topology and refer to it as the \textit{full shift} on alphabet $\mathcal{A}$. This makes $\Sigma(\mathcal{A})$ into a compact, metrizable space. For each $g \in G$, we define $\sigma^g : \Sigma(\mathcal{A}) \to \Sigma(\mathcal{A})$ by the following rule: for all $x \in \mathcal{A}^G$ and $h \in G$, let
\begin{equation*}
\sigma^gx (h) = x(g^{-1} h).
\end{equation*}
We note that for each $g \in G$ the map $\sigma^g : \Sigma(\mathcal{A}) \to  \Sigma(\mathcal{A})$ is a homeomorphism, and this collection of maps $\{\sigma^g : g \in G\}$ defines an action of $G$, called the \textit{shift action}. A set $Z \subset  \Sigma(\mathcal{A})$ is a \textit{subshift} if it is closed and shift-invariant (meaning that $\sigma^g(Z) = Z$ for all $g \in G$).

We refer to elements of $\mathcal{A}^S$ as \textit{patterns}, and if $w \in \mathcal{A}^S$, then we say that $w$ has \textit{shape} $S$. If $S' \subset S \subset G$ and $w \in \mathcal{A}^S$, then we let $w(S')$ denote the pattern $u \in \mathcal{A}^{S'}$ such that $u(s) = w(s)$ for all $s \in S'$. 
Additionally, for patterns $u \in \mathcal{A}^E$ and $v \in \mathcal{A}^F$ such that $u(g) = v(g)$ for all $g \in E \cap F$, we let $u \cup v$ be the pattern in $\mathcal{A}^{E \cup F}$ such that $(u \cup v)(g) = u(g)$ for all $g \in E$ and $(u \cup v) (g) = v(g)$ for all $g \in F$. If $E \cap F  = \varnothing$, then we may write $u \sqcup v$. 

Let $Z$ be a subshift. For a given pattern $w \in \A^S$ we define the \textit{cylinder set} $[w] = \{z\in Z : z(S) = w\}$.
Next we define 
\begin{equation*}
\mathcal{L}_S(Z) = \{ z(S) : z \in Z\},
\end{equation*}
and
\begin{equation*}
\mathcal{L}(Z) = \bigcup_{S \subset G} \mathcal{L}_S(Z).
\end{equation*}
We refer to $\mathcal{L}(Z)$ as the \textit{language} of $Z$. Note that we allow $S = \varnothing$, so that the empty word $\lambda$ is in $\mathcal{L}(Z)$, and we allow $S$ to be infinite, which differs from some definitions of the language of a subshift. We also define an action of $G$ on $\mathcal{L}(Z)$ as follows: for $g \in G$ and $w \in \mathcal{A}^S$, we let $g \cdot w$ be the element of $\mathcal{A}^{gS}$ such that $(g \cdot w)(gs) = w(s)$ for all $s\in S$.

Let $\mathcal{A}$ and $\mathcal{B}$ be alphabets, and let $Z$ be a subshift on $\mathcal{A}$.  For any nonempty finite subset $R \subset G$, a \textit{block map} with shape $R$ is any function $\Phi : \mathcal{L}_R(Z) \to \mathcal{B}$. To ease notation, if $S \subset G$ and $u \in \mathcal{A}^{SR}$, then we may let $\Phi(u)$ denote the pattern with shape $S$ such that $\Phi(u)(s) = \Phi(s^{-1} \cdot u(sR))$ for all $s \in S$. A function $\phi : Z \to \Sigma(\mathcal{B})$ is called a  \textit{sliding block code} if there exists a block map $\Phi$ with some shape $R$ such that for all $z \in Z$ and all $g \in G$, we have
\begin{equation*}
\bigl(\phi (z)\bigr) (g) = \Phi \bigl( g^{-1} \cdot z( gR) \bigr).
\end{equation*}
{If a block map $\Phi$ can be selected with shape $R = \{e\}$ (in other words, if $\Phi$ may be viewed as a function on the alphabets $\Phi : \mathcal{A} \to \mathcal{B}$), then $\phi$ is said to be a \textit{$1$-block code}.} 

A map $\phi : Z \to \Sigma(\mathcal{B})$ is shift-commuting if $\sigma^g \circ \phi = \phi \circ \sigma^g$ for all $g \in G$. By the Curtis-Hedlund-Lyndon Theorem \cite[Theorem~1.8.1]{springer_book}, a map $\phi : Z \to \Sigma(\mathcal{B})$ is a homomorphism (continuous and shift-commuting) if and only if $\phi$ is a sliding block code.

Suppose $X$ and $Y$ are subshifts and $\phi : X \to Y$ is a homomorphism. We say that $\phi$ is a \textit{factor map} whenever it is surjective, and we say that $\phi$ is an \textit{embedding} whenever it is injective. 

For a finite shape $K \subset G$ and a collection $\mathcal{F} \subset \mathcal{A}^K$, we let
\begin{equation*}
X_{\mathcal{F}} = \bigl\{ x\in \mathcal{A}^G : \forall g \in G, \, \sigma^g x (K) \notin \mathcal{F}\bigr\}.
\end{equation*}
A subshift $X \subset \mathcal{A}^G$ is a \textit{shift of finite type} (SFT) if there exists a finite shape $K \subset G$ and a collection $\mathcal{F} \subset \mathcal{A}^K$ such that $X = X_{\mathcal{F}}$.

A subshift $X$ is \textit{strongly irreducible} (SI) if there exists a finite set $K \subset G$ such that for all shapes $E, F \subset G$ and all patterns $u \in \mathcal{L}_E(X)$ and $v \in \mathcal{L}_F(X)$, if $EK \cap F = \varnothing$, then $u \cup v \in \mathcal{L}_{E \cup F}(X)$. 

For any point $x \in \mathcal{A}^G$ and any $g \in G \setminus \{e\}$, we say that $x$ has $g$ as a period if $\sigma^g x = x$. On the other hand, if $\sigma^g x = x$ implies that $g = e$, then we say that $x$ is \textit{aperiodic}. Furthermore, if a subshift $X$ consists entirely of aperiodic points, then we refer to the subshift as aperiodic. We note that some authors refer to such subshifts as \textit{strongly aperiodic}. Lastly, if there exists $g \in G \setminus \{e\}$ such that every point of $x$ has $g$ as a period, then we say that $X$ has a global period.

Next we define the finite extension property for subshifts over countable groups $G$. We note that this definition extends the definition of the finite extension property for $\mathbb{Z}^d$-subshifts given in \cite[Definition~2.11]{briceno_mcgoff_pavlov}. 

\begin{definition}
Suppose $G$ is a countable group. A $G$-subshift $X$ with alphabet $\mathcal{A}$ has the \textit{finite extension property} (FEP) if there exists a finite set $K \subset G$ containing the identity element $e$ and a collection $\mathcal{F} \subset \mathcal{A}^K$ such that $X = X_{\mathcal{F}}$ and for any $F \subset G$ and any pattern $w \in \mathcal{A}^{F}$, if there exists $w' \in \mathcal{A}^{FK}$ such that $w'(F) = w(F)$ and $g^{-1} \cdot w(gK) \notin \mathcal{F}$ for all $g \in F$, then there exists $x \in X$ such that $x(F) = w(F)$. 
\end{definition}

The finite extension property is a very strong property. Indeed, if $X$ has the FEP, then $X$ must be a strongly irreducible $G$-SFT (Proposition \ref{Prop:SIandSFT}).  Additionally, we note that the FEP is invariant under conjugacy (Proposition \ref{Prop:InvariantUnderConjugacy}).

\begin{remark} \label{Remark:FEP-Restatement}
Here we mention a useful restatement of the FEP.
Suppose $G$ is a countable group, the set $K \subset G$ is finite and contains the identity element $e$,  $\mathcal{F} \subset \mathcal{A}^{K}$, and $X = X_{\mathcal{F}}$. Then the following are equivalent:
\begin{itemize}
\item[(i)] $X$ has the finite extension property with shape $K$ and forbidden set $\mathcal{F}$;
\item[(ii)] for any $S \subset G$ and any $w \in \mathcal{A}^S$, if $w(S)$ contains no patterns from $\mathcal{F}$ (meaning $g^{-1} \cdot w(gK) \notin \mathcal{F}$ for all $g \in \int_K(S)$), then $w(\int_K(S)) \in \mathcal{L}(X)$.
\end{itemize}
\end{remark}

\subsection{Entropy and invariant measures}

Let $G$ be an amenable group, and fix a F{\o}lner sequence $\{F_n\}_{n=1}^{\infty}$. Then 
for any nonempty $G$-subshift $X$, the topological entropy of $X$ is given by
\begin{equation*}
h(X) = \lim_{n \to \infty} \frac{1}{|F_n|} \log |\mathcal{L}_{F_n}(X)|.
\end{equation*}
We note that this limit exists and is independent of the choice of F{\o}lner sequence (see the book \cite{kerr_li}). A subshift $X$ is \textit{entropy minimal} if all of its proper subsystems have strictly less entropy, i.e., for all subshifts $Y \subsetneq X$, we have $h(Y) < h(X)$.

Let $X$ be a $G$-subshift.
We let $M(X,\sigma)$ denote the set of Borel probability measures $\mu$ on $X$ such that $\mu(\sigma^g(A)) = \mu(A)$ for all $g \in G$ and all Borel sets $A$. Recall that for any pattern $u \in \mathcal{L}_F(X)$, we let $[u]$ denote the cylinder set defined by the condition $x(F) = u$.
 For such a measure $\mu$,  the entropy of $\mu$ is given by
\begin{equation*}
h(\mu) = \lim_{n \to \infty} \frac{1}{|F_n|} \sum_{u \in \mathcal{L}_{F_n}(X)} - \mu([u]) \log \mu([u]),
\end{equation*}
where again the limit exists and is independent of the choice of F{\o}lner sequence (\cite{kerr_li}).
By the Variational Principle (see \cite{kerr_li}), we have
\begin{equation*}
h(X) = \sup_{\mu \in M(X,\sigma)} h(\mu).
\end{equation*}
Moreover, since $X$ is a subshift, it is known that there exists $\mu \in M(X,\sigma)$ such that $h(\mu) = h(X)$ (see \cite{kerr_li}). Such measures are called \textit{measures of maximal entropy}.

\section{Properties of subshifts with the FEP} \label{Sect:FEP}

In this section we establish some basic properties of subshifts with the FEP. First, we show that the FEP is invariant under conjugacy.

\begin{prop} \label{Prop:InvariantUnderConjugacy}
Suppose $G$ is a countable group, $X$ and $Y$ are $G$-subshifts, and $X$ and $Y$ are topologically conjugate. Then $X$ has the FEP if and only if $Y$ has the FEP.
\end{prop}
\begin{proof}
Suppose $X$ has the FEP with shape $K$ and forbidden set $\mathcal{F} \subset \mathcal{A}_X^K$. Suppose $\phi : X \to Y$ is a topological conjugacy. Suppose $\phi$ has a block map $\Phi_0$ with block shape $R_0$ and $\phi^{-1}$ has a block map $\Phi_1$ with block shape $R_1$. We suppose without loss of generality that $R_0$ and $R_1$ each contain the identity.

Let $K' = R_0KR_1$, and let $\mathcal{F}' \subset \mathcal{A}_Y^{K'}$ be the set of all patterns $u$ such that there exists $g \in R_0$ with the property that $\Phi_1(g^{-1} \cdot u(gK)) \in \mathcal{F}$. We claim that $Y$ has the FEP with shape $K'$ and forbidden set $\mathcal{F}'$. 

To establish this claim, let $E \subset G$, and suppose that $u \in \mathcal{A}_Y^{EK'}$ has the property that for all $g \in E$, we have $g^{-1} \cdot u(gK') \notin \mathcal{F}'$. Let $v = \Phi_1(u) \in \mathcal{A}_X^{ER_0K}$. By the definition of $\mathcal{F}'$ and our assumption on $u$, for all $g \in ER_0$, we have $g^{-1} \cdot v(gK) = \Phi_1(g^{-1} \cdot u(g K') ) \notin \mathcal{F}$. Since $X$ has the FEP with shape $K$ and forbidden set $\mathcal{F}$, there exists a point $x \in X$ such that $x(ER_0) = v(ER_0)$. Now let $y = \phi(x) \in Y$. Furthermore, observe that $y(E) = \Phi_0(v(ER_0)) = \Phi_0( \Phi_1(u(ER_0R_1))) = u(E)$, where we have used that $\Phi_0$ and $\Phi_1$ are the block codes for the inverse maps $\phi$ and $\phi^{-1}$, respectively. Hence $u(E) \in \mathcal{L}(Y)$, and thus we have established that $Y$ has the FEP, which concludes the proof.
\end{proof}

Next we establish that any subshift with the FEP is a strongly irreducible SFT.

\begin{prop} \label{Prop:SIandSFT}
If $X$ is a $G$-subshift with the FEP, then $X$ is a strongly irreducible $G$-SFT.
\end{prop}
\begin{proof}
Suppose $X$ has the FEP witnessed by a shape $K$ (containing the identity) and a set of forbidden patterns $\mathcal{F} \subset \mathcal{A}^K$. Passing to a larger shape (with a corresponding set of forbidden patterns) if necessary, we assume without loss of generality that $K$ is symmetric. Note that $X$ is a $G$-SFT since $X = X_{\mathcal{F}}$.

Next we claim that $X$ is strongly irreducible with shape $K^4$. Let $E,F \subset G$ be shapes such that $EK^4 \cap F = \varnothing$. Let $u \in \mathcal{L}_E(X)$ and $v \in \mathcal{L}_F(X)$. Since $u$ and $v$ are in the language of the subshift, there exist extensions $u' \in \mathcal{L}_{EK}(X)$ and $v' \in \mathcal{L}_{FK}(X)$. Note that $EK \cap FK = \varnothing$, since $EK^2 \cap F = \varnothing$. Therefore $w = u' \cup v'$ is a well-defined pattern on $EK \cup FK = (E \cup F) K$. We claim that every pattern in $w$ of shape $gK$ is in the language of $X$, and therefore $w$ contains no (translates of) patterns from $\mathcal{F}$. Let $gK \subset EK \cup FK$. If $gK \subset EK$, then $w(gK) = u'(gK) \in \mathcal{L}(X)$, and similarly if $gK \subset FK$, then $w(gK) = v'(gK) \in \mathcal{L}(X)$. If $gK \cap EK \neq \varnothing$ and $gK \cap FK \neq \varnothing$, then $g \in EK^2$ and there exists $f \in F$ such that $f \in gK^2 \subset EK^4$, which is impossible since $EK^4 \cap F = \varnothing$. Thus, $w$ contains no (translates of) forbidden patterns from $\mathcal{F}$. By the FEP, $u \cup v = w(E \cup F) \in \mathcal{L}(X)$, and hence there is a point $x \in X$ such that $x(E) = u$ and $x(F) = v$, as desired.
\end{proof}

In order to prove that the FEP implies entropy minimality (Proposition \ref{Prop:EntropyMinimality} below), we make use of a result of \cite{hedges} that refers to \textit{extender sets}. Given a pattern $u \in \mathcal{L}_S(X)$, its extender set $E(u)$ is defined as the set of patterns $\eta \in \mathcal{L}_{S^c}(X)$ such that $u \cup \eta \in X$. First we require a lemma that gives a sufficient condition for equality of extender sets of two patterns.

\begin{lemma}
Suppose $X$ is an SFT with forbidden shape $K$ containing the identity. Let $S \subset G$, and let $\partial_{K^{-1}K}(S) = SK^{-1}K \setminus S$. If $u,v \in \mathcal{L}_{SK^{-1}K}(X)$ and $u(\partial_{K^{-1}K}(S)) = v(\partial_{K^{-1}K}(S))$, then $E(u) = E(v)$.
\end{lemma}
\begin{proof}
First we claim that for every $g \in G$, either $gK \subset S^c$ or $gK \subset S \cup \partial_{K^{-1}K}(S)$. To see this, suppose that $gK$ is not contained in $S^c$. Then there exists $f \in gK \cap S$. Hence $g \in S K^{-1}$, so $gK \subset S K^{-1} K = S \cup \partial_{K^{-1}K}(S)$.

Let $\eta \in E(u)$. Let $x = u \cup \eta \in X$, and let $y = v \cup \eta \in \mathcal{A}^G$. 
Now we claim that $\eta \in E(v)$. Let $g \in G$. If $gK \subset S \cup \partial_{K^{-1}K}(S) = SK^{-1}K$, then $y(gK) =  v(gK) \in \mathcal{L}(X)$. Now suppose $gK \subset S^c = \partial_{K^{-1}K}(S) \cup (SK^{-1}K)^c$. Then $y(gK) = v(\partial_{K^{-1}K}(S)) \cup \eta((SK^{-1}K)^c) = v(\partial_{K^{-1}K}(S)) \cup \eta((SK^{-1}K)^c) \in \mathcal{L}(X)$. Since $X$ is an SFT with forbidden shape $K$, we obtain that $y \in X$, and hence $\eta \in E(v)$. Repeating the argument with the roles of $u$ and $v$ reversed, we conclude that $E(u) = E(v)$.  
\end{proof}

It is well-known that strongly irreducible $\mathbb{Z}^d$-SFTs are entropy minimal \cite[Lemma 4.1]{meester_steif}. 
The following proposition extends this result to the set of countable amenable groups $G$. Our proof follows the same strategy as the proof in \cite{meester_steif}, but we use a recent result of \cite{hedges} in place of the result of  \cite{burton_steif_NewResults} concerning measures of maximal entropy.

\begin{prop} \label{Prop:EntropyMinimality}
If $G$ is a countable amenable group and $X$ is a strongly irreducible $G$-SFT, then $X$ is entropy minimal. In particular, if $X$ has the FEP, then $X$ is entropy minimal.
\end{prop}
\begin{proof}
Suppose $K$ is a finite set witnessing both the strong irreducibility and the SFT property for $X$.  
Assume without loss of generality that $K$ contains the identity and $K^{-1} = K$.
Let $\mathcal{F} \subset \mathcal{A}^K$ be a collection of forbidden patterns that defines $X$. 
Let $F \subset G$ and suppose for contradiction that forbidding $u_* \in \mathcal{L}_F(X)$ from $X$ yields another SFT $Y$ such that $h(Y) = h(X)$. Let $\mu$ be a measure of maximal entropy on $Y$. Let $y \in Y$ be in the topological support of $\mu$. Let $v = y(FK^3)$. By the choice of $K$ (witnessing that $X$ is SI), there exists $x \in X$ such that $x(F) = u_*$ and $x((FK)^c) = y((FK)^c)$. Let $u = x(FK^3)$ and $v = y(FK^3)$.

Let $S = FK$, and note that $u( \partial_{K^{-1}K}(S)) = v(\partial_{K^{-1}K}(S))$. Then by the previous lemma, we obtain that the extender sets of $u$ and $v$ satisfy $E(u) = E(v)$.
Now by \cite[Corollary~3.19]{hedges}, we have that $\mu([u]) = \mu([v])$. Then
\begin{equation*}
\mu([u_*]) \geq \mu( [u]) = \mu([v]) >0,
\end{equation*}
where the first inequality is due to monotonicity of measure and the second inequality holds since $y$ is in the topological support of $\mu$ and $[v]$ is an open set containing $y$.
However, the previous display leads to a contradiction, since $\mu$ is supported on $Y \subset X \setminus [u^*]$ and thus $\mu([u_*]) = 0$.
\end{proof}

\section{Preliminary quasi-tiling results} \label{Sect:Lemmas}

Quasi-tilings have played a significant role in recent work on symbolic dynamics (e.g., see \cite{bland,
bland2023subsystem,
downarowicz_huczek_dynamical-quasitilings,
downarowicz2019tilings,
downarowicz_zhang-symb-ext,
frisch2017symbolic,
huczek_kopacz,
mcgoff2021ubiquity}). 
Here we present the following general definition of quasi-tilings, which is suitable for our purposes.
\begin{definition}
    Let $\mathcal{S} = \{S_1,\dots,S_r\}$ be a finite collection of nonempty, finite subsets of $G$, referred to as shapes. A \textit{quasi-tiling} of $G$ is an assignment of each shape $S \in \mathcal{S}$ to a subset $C_S \subset G$, referred to as the center set for $S$, such that the collection of sets $\{C_S : S \in \mathcal{S}\}$ is pairwise disjoint and the map $(S,c) \to cS$ is injective over $\{(S,c) : S \in \mathcal{S} \text{ and } c \in C_S\}$. For $S \in \mathcal{S}$ and $c \in C_S$, the set $cS$ is called the tile with center $c$ and shape $S$.
\end{definition}

In this work we only require quasi-tilings with a single shape, and therefore we introduce some more specific notation.
Recall that there is a standard bijection between subsets of $G$ and elements of the full shift $\Sigma(\{0,1\})$, given by mapping a subset $\mathcal{C} \subset G$ to its characteristic function $\chi_{\mathcal{C}}$ (viewed as an element of $\Sigma(\{0,1\})$). Also, we note that $\max(\chi_{A}, \chi_B) = \chi_{A \cup B}$ and $\min(\chi_A, \chi_B) = \chi_{A \cap B}$, and therefore the operations of taking unions or intersections of sets can be viewed as  $1$-block codes from $\Sigma \times \Sigma$ to $\Sigma$. 

\begin{definition}
Let $F$ be a finite subset of $G$. For any $z = \chi_{\mathcal{C}} \in \Sigma(\{0,1\})$, we say that $z$ is $F$-\textit{disjoint} if the collection $\{c F\}_{c \in \mathcal{C}}$ is pairwise disjoint, and $z$ is $F$-\textit{covering} if the collection $\{c F \}_{c \in \mathcal{C}}$ covers $G$. Furthermore, for any subshift $Z \subset \Sigma(\{0,1\})$, we say that $Z$ is $F$-disjoint if each element of $Z$ is $F$-disjoint, and similarly $Z$ is $F$-covering if each element of $Z$ is $F$-covering.
\end{definition}

The use of quasi-tilings for the purposes of studying the dynamics of amenable group actions dates back to the fundamental construction of Ornstein and Weiss \cite[I.\textsection 2~Theorem~6]{ornstein_weiss}. A dynamical version of this construction was recently obtained by Downarowicz and Huczek \cite[Lemma~3.4]{downarowicz_huczek_dynamical-quasitilings}, which derives a system of quasi-tilings as a factor of a given aperiodic subshift. In the following proposition we construct quasi-tilings suitable for our purposes, adapting the first part of the proof of \cite[Lemma~3.4]{downarowicz_huczek_dynamical-quasitilings} (here we require strict disjointness of the tiles, whereas in \cite{downarowicz_huczek_dynamical-quasitilings} some small overlap is allowed). We provide a detailed proof for clarity and convenience.

\begin{prop} \label{Prop:Quasi}
Let $G$ be a countable group. Suppose $X$ is a nonempty aperiodic $G$-subshift and $F \subset G$ is a nonempty finite subset. Then there exists a homomorphism  $\phi : X \to \Sigma(\{0,1\})$ such that $\phi(X)$ is $F$-disjoint and $FF^{-1}$-covering.
\end{prop}

\begin{proof}
Let $x_0\in X$ be arbitrary. Since $x_0$ is aperiodic, there must exist a finite subset $S_0$ such that $x_0(S_0) \neq \sigma^fx_0(S_0)$ for all $f\in FF^{-1}\setminus\{e\}$. Then choose $S_1 = FF^{-1}S_0$ and let $u_{x_0} = x_0(S_1)$. For any $x_1 \in [u_{x_0}]$ it must therefore hold that $x_1(S_0) \neq \sigma^fx_1(S_0)$ for all $f\in FF^{-1}\setminus\{e\}$, and hence  $[u_{x_0}] \cap \sigma^{f^{-1}}[u_{x_0}] = \varnothing$ for all $f\in FF^{-1}\setminus\{e\}$. Repeat this construction for all possible $x = x_0$; the collection $\{[u_x] : x\in X\}$ is an open cover of $X$, and therefore by compactness we may find a finite subcover $\mathcal{U} = \{[u_1],\ldots,[u_n]\}$ for some $n \geq 1$.

Now let $x\in X$ be fixed. We shall construct a corresponding point $z \in \Sigma(\{0,1\})$. Let $\mathcal{C}_0 = \varnothing$ and for each $i\in[1,n]$ inductively define \[
    \mathcal{C}_{i} = \{g\in G : \sigma^{g^{-1}}x\in[u_i] \text{ and } gFF^{-1} \cap \mathcal{C}_{i-1} = \varnothing\} \cup \mathcal{C}_{i-1}.
\] Let $z = \chi_{\mathcal{C}_n} \in \Sigma(\{0,1\})$. We claim that the map $\phi : X \to \Sigma(\{0,1\})$ given by $\phi(x) = z$ satisfies the conditions in the conclusion of the theorem. It remains to prove that $z$ is $F$-disjoint, $z$ is $FF^{-1}$-covering, and $\phi$ is a sliding block code.

To prove that $z$ is $F$-disjoint, let $g_1 \neq g_2 \in \mathcal{C}_n$ be arbitrary. We claim that $g_1 F \cap g_2 F = \varnothing$. Suppose to the contrary that there exists some $f_1, f_2 \in F$ such that $g_1f_1 = g_2f_2$ (note that $f_1 \neq f_2$).  Given that $g_1, g_2 \in \mathcal{C}_n$, there must exist a minimal $n_1, n_2 \geq 1$ such that $g_1 \in \mathcal{C}_{n_1}$ and $g_2 \in \mathcal{C}_{n_2}$. Without loss of generality, assume that $n_1 \leq n_2$. If $n_1 < n_2$ (in which case $\mathcal{C}_{n_1} \subset \mathcal{C}_{n_2-1}$ by construction), then by construction it must hold that $g_2FF^{-1}\cap\mathcal{C}_{n_2-1} = \varnothing$, but this contradicts the fact that $g_2f_1f_2^{-1} = g_1 \in \mathcal{C}_{n_1}$. 
Otherwise, if $n_1 = n_2$, then $\sigma^{g_1^{-1}}x, \sigma^{g_2^{-1}}x \in [u_{n_1}]$. Moreover, from $g_1f_1 = g_2f_2$ we see that $g_2^{-1} = f_2f_1^{-1}g_1^{-1}$. Set $x_0 = \sigma^{g_1^{-1}}x \in [u_{n_1}]$, and notice that $\sigma^{f_2f_1^{-1}}x_0 =\sigma^{g_2^{-1}}x \in[u_{n_1}]$; however, this contradicts the fact that $[u_{n_1}] \cap \sigma^{f^{-1}}[u_{n_1}] = \varnothing$ for all $f\in FF^{-1}\setminus\{e\}$ by construction of $\mathcal{U}$. We derive a contradiction in either case, and therefore $g_1F\cap g_2F = \varnothing$ must hold. We conclude that $z$ is $F$-disjoint.

To prove that $z$ is $FF^{-1}$-covering, let $g\in G$ be arbitrary. Because $\mathcal{U}$ covers $X$, there must exist an $i \in [1,n]$ such that $\sigma^{g^{-1}}x \in [u_i]$; choose and fix a minimal such $i$. If $g\in \mathcal{C}_n$, then note that $g\in gFF^{-1}$ (since $e\in FF^{-1}$), and hence $g \in \cup_{c \in C_n} c FF^{-1}$, as desired. Otherwise, we must have that $\sigma^{g^{-1}}x \in [u_i]$ and $gFF^{-1}\cap \mathcal{C}_{i-1} \neq \varnothing$. In this case, there exists a $c\in \mathcal{C}_{i-1}$ such that $c\in gFF^{-1}$, which gives $g\in cFF^{-1}$, as desired. In either case, we see that there must exist a $c\in \mathcal{C}_n$ such that $g\in cFF^{-1}$, and we conclude that $z$ is $FF^{-1}$-covering.

Finally, let us show that $\phi : X \to \Sigma(\{0,1\})$ is a sliding block code. Suppose the patterns $\{u_1,\ldots, u_n\}$ that make up $\mathcal{U}$ have corresponding shapes $S_1, \ldots, S_n \subset G$. 
Let $R_0 = \varnothing$, and for each $i \in [1,n]$, inductively define $R_i = S_i \cup FF^{-1}R_{i-1}$. 
We shall argue inductively that for each $i \geq 1$, there is a block map $\Phi_i : \mathcal{L}_{R_i}(X)\to\{0,1\}$ that witnesses the sliding block code property for the map $x \mapsto \chi_{\mathcal{C}_i}$. 
We note that the existence of such a block map is equivalent to the existence of a procedure which correctly decides, for each fixed $x$ and $g$, whether $g\in \mathcal{C}_i$ based solely on the pattern $\sigma^{g^{-1}}x(R_i)$. Below we describe such a procedure.

For $i = 1$, note that $g \in \mathcal{C}_1$ if and only if $\sigma^{g^{-1}}x\in [u_1]$, which can be determined from the pattern $\sigma^{g^{-1}}x(R_1) = \sigma^{g^{-1}}x(S_1)$. Now suppose for induction that $i > 1$ is fixed and there is a procedure to decide whether $g\in \mathcal{C}_{i-1}$ from $\sigma^{g^{-1}}x(R_{i-1})$ for each $x$ and $g$. Note that $g\in \mathcal{C}_i$ if and only if $\sigma^{g^{-1}}x\in [u_i]$ and $gFF^{-1}\mathcal{C}_{i-1} = \varnothing$. The first condition can be decided from $\sigma^{g^{-1}}x(S_i)$. To decide the second, it is sufficient to check whether $gf_1f_2^{-1}\in\mathcal{C}_{i-1}$ for each $f_1,f_2\in F$. By induction, this condition can be determined from $\sigma^{(gf_1f_2^{-1})^{-1}}x(R_{i-1})$, and therefore knowledge of $\sigma^{g^{-1}}(FF^{-1}R_{i-1})$ suffices. Hence it is possible to determine whether $g\in \mathcal{C}_i$ based on knowledge of $\sigma^{g^{-1}}x(R_i) = \sigma^{g^{-1}}x(S_i \cup FF^{-1}R_{i-1})$, completing the induction and the proof.
\end{proof}

The following key lemma uses the doubling property to control the number of balls with large disjoint interiors that can intersect at a single group element. We apply this lemma in our proof of Theorem \ref{Theorem:Main-Result} to bound the number of stages  necessary for our inductive argument.

\begin{lemma} \label{Lemma:BoundedDegree}
Suppose $G$ is a countable group and $K$ is a finite symmetric set that contains the identity and satisfies the doubling property (for the subgroup $\langle K \rangle$) with constant $C_0 = C_0(K)$. Let $m_0 \geq C_0^2$ and $n_0 \in \N$. 
Suppose $Z \subset \Sigma(\{0,1\})$ is a subshift that is $K^{n_0}$-disjoint. Then for all $z = \chi_{\mathcal{C}} \in Z$, 
for all $n \in [0,n_0]$, and for all $g \in G$, we have
\begin{equation*}
\bigl| \{ c \in \mathcal{C} : \dist_K(g,c K^{2n_0}) \leq n \} \bigr|  \leq m_0.
\end{equation*}
\end{lemma}
\begin{proof}
Let $S = \{ c \in \mathcal{C} : \dist_K(g, cK^{2n_0}) \leq n\}$, and suppose $|S| = m$. Then for each $c \in S$, we have $\rho_K(g,c) \leq 2n_0 + n \leq 3n_0$ (where we have used the hypothesis that $n \leq n_0$).
Since $z$ is $K^{n_0}$-disjoint, the collection $\{c K^{n_0}\}_{c \in S}$ is a disjoint collection of balls of radius $n_0$ with respect to $\rho_K$, and they are contained in the ball $g K^{4n_0}$. Hence we have
\begin{equation*}
m |B_{n_0}(K)| \leq |B_{4n_0}(K)|.
\end{equation*}
Dividing by $|B_{n_0}(K)|$ and using the hypotheses on $C_0$ and $m_0$, we conclude that
\begin{align*}
m & \leq \frac{ |B_{4n_0}(K)|}{ |B_{n_0}(K)| } \\
& \leq  \frac{ |B_{2n_0}(K)|}{ |B_{n_0}(K)| } \cdot  \frac{ |B_{4n_0}(K)|}{ |B_{2n_0}(K)| } \\
& \leq C_0^2 \\
& \leq m_0.
\end{align*}
\end{proof}

\section{Proof of Theorem \ref{Theorem:Main-Result}} \label{Sect:Proof}

Here we begin the proof of Theorem \ref{Theorem:Main-Result}. For convenience, we provide a restatement of the theorem.

\begin{main-result}
Let $G$ be a countable group such that every finitely generated subgroup of $G$ has polynomial growth. 
Let $X$ be an aperiodic $G$-subshift, and let $Y$ be a nonempty $G$-subshift with the finite extension property. Then there exists a homomorphism $\phi : X \to Y$. 
\end{main-result}

Suppose $G$, $X$ and $Y$ satisfy the hypotheses of the theorem.
Let $K \subset G$ be a finite symmetric set containing the identity and witnessing both the FEP and the strong irreducibility of $Y$. Let $\mathcal{F} \subset \mathcal{A}^K$ be a set of forbidden patterns witnessing the FEP of $Y$.  Since the group generated by $K$ is a finitely generated subgroup of $G$, it has polynomial growth (by hypothesis) and therefore satisfies the doubling property. Let $C_0 = C_0(K)$ be a constant witnessing the doubling property on $\langle K \rangle$. 
Choose a natural number $m_0$ such that
$m_0 \geq C_0^2$, 
and let $n_0 = 3m_0$. 

Let $F = K^{n_0}$, and note that by the symmetry of $K$ we have $F^{-1} = K^{n_0}$. By Proposition \ref{Prop:Quasi}, there exists a homomorphism $\phi_0 : X \to \Sigma(\{0,1\})$ such that $\phi_0(X)$ is $K^{n_0}$-disjoint and $K^{2n_0}$-covering. Let $Z = \phi_0(X)$.

The remainder of the proof involves constructing a homomorphism $\phi_1 : Z \to Y$, which we organize as follows. First, given $z \in Z$, we inductively construct pairs of sets $(U_1,V_1),\dots,(U_{m_0},V_{m_0})$, where $U_m,V_m \subset G$ for each $m = 1,\dots,m_0$ and $U_{m_0} = V_{m_0} = G$. Then we use the FEP to construct patterns on each $U_m$ and $V_m$, and we define $\phi_1(z)$ to be the pattern constructed on $V_{m_0}$.

\subsection{The sets $U_m$ and $V_m$}

The construction of the sets $U_m$ and $V_m$ proceeds by induction on $m$, with $m_0$ stages in total, but first we define some helpful notation.
For each $z = \chi_{\mathcal{C}} \in Z$, $g \in G$, and $n \geq 0$, let
\begin{equation*}
\deg^z_{n}(g) = | \{ c \in \mathcal{C} : \dist(g,cK^{2n_0}) \leq n\}|.
\end{equation*}
We easily observe that $\deg^z_n(g) \leq \deg^z_{n+1}(g)$, and 
since $z$ is $K^{2n_0}$-covering, we have $\deg^z_{0}(g) \geq 1$ for all $g \in G$.
Next, for each $n \geq 0$ and $m \geq 0$, we define the set
\begin{equation*}
E^z_{m,n} = \{ g \in G : \deg^z_{n}(g) \leq m\}.
\end{equation*}
Observe that by Lemma \ref{Lemma:BoundedDegree}, we have $E^z_{m_0,n_0} = G$ for all $z \in Z$.

Now let $z = \chi_{\mathcal{C}} \in Z$ {be fixed}, and 
let us inductively construct some sets $U_m = U_m(z)$ and $V_m = V_m(z)$ for $m = 1,\dots, m_0$. To ease notation, let $E_{m,n} = E_{m,n}^z$ and $\deg_n = \deg_n^z$. For all $c \in \mathcal{C}$, we let
\begin{equation*}
T_c = cK^{2n_0} \cap E_{1,2}.
\end{equation*}
For $m=1$, we set
\begin{align*}
U_1 & = \bigcup_{c \in \mathcal{C}} T_c,
\end{align*}
and
\begin{align*}
V_1 & = \int_K(U_1).
\end{align*}
Now suppose by induction that $1 < m \leq m_0$ and $U_{m-1}$ and $V_{m-1}$ have been defined. For notation, let $\mathcal{C}_m = \mathcal{C}_m(z)$ denote the set of subsets of $\mathcal{C}$ of cardinality exactly $m$. For each $S \in \mathcal{C}_m$, we let
\begin{equation*}
\tilde{T}_{S} = \bigcap_{c \in S} c K^{2n_0 + 3m-3},
\end{equation*}
and
\begin{equation*}
T_{S} = \tilde{T}_S \cap E_{m,3m-1}.
\end{equation*}
Then we define
\begin{align*}
U_{m} & = \left( \bigcup_{S \in \mathcal{C}_m} T_{S} \right) \cup V_{m-1},
\end{align*}
and
\begin{align*}
V_{m} & = \int_K(U_{m}).
\end{align*}
This completes the inductive definition of the sets $U_m$ and $V_m$ for $m = 1,\dots,m_0$. For an illustration of the construction, see Figure~\ref{fig:construction}. The following three lemmas describe some properties of these sets.

\begin{figure}[hbt]
    \centering
    \begin{subfigure}[t]{.4\textwidth}
    \centering
    \includegraphics[width=5cm]{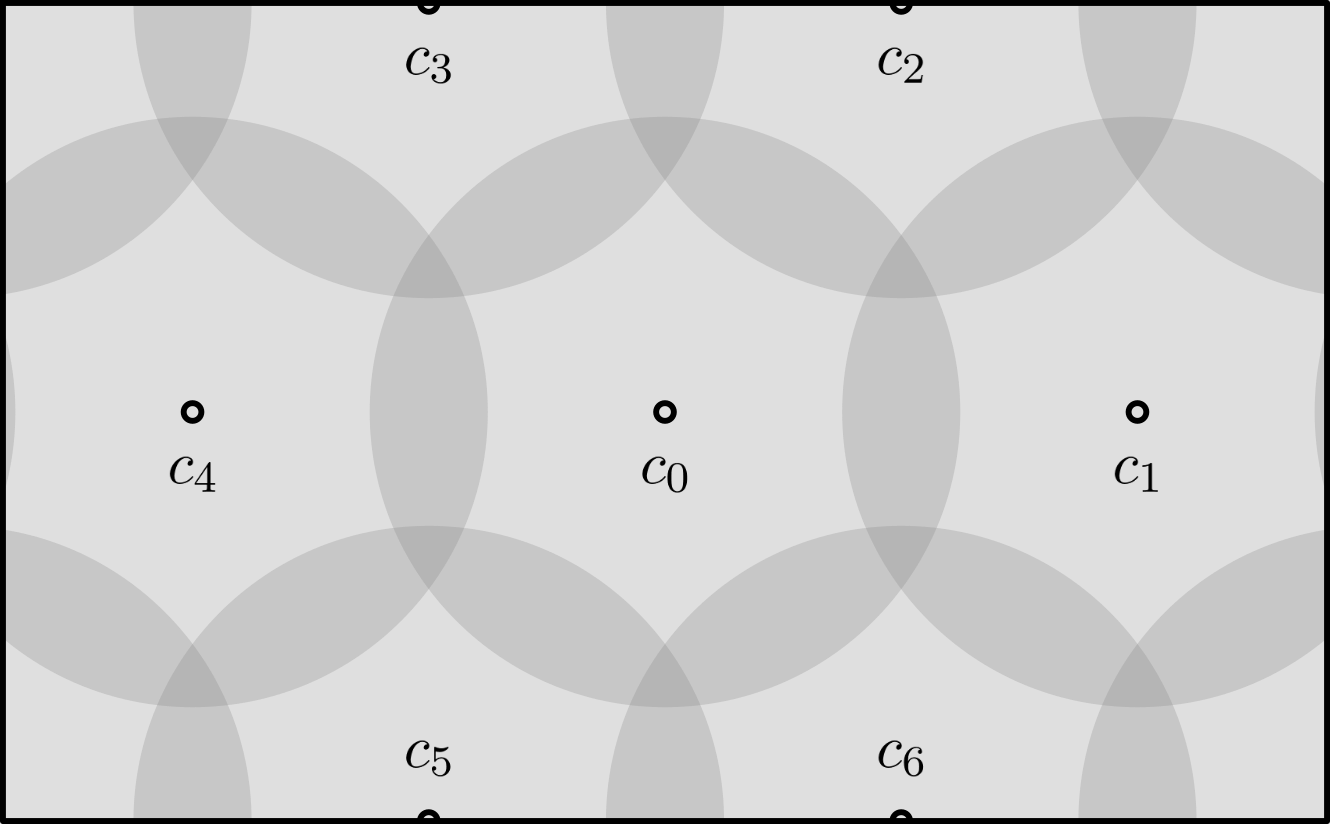}
    \caption{A hypothetical quasi-tiling $z$. The centers $\mathcal{C}$ are marked with points and the tiles of shape $K^{2n_0}$ are shown as light grey circles.}
    \label{fig:construction_1}
    \end{subfigure}
    \hspace{1cm}
    \begin{subfigure}[t]{.4\textwidth}
    \centering
    \includegraphics[width=5cm]{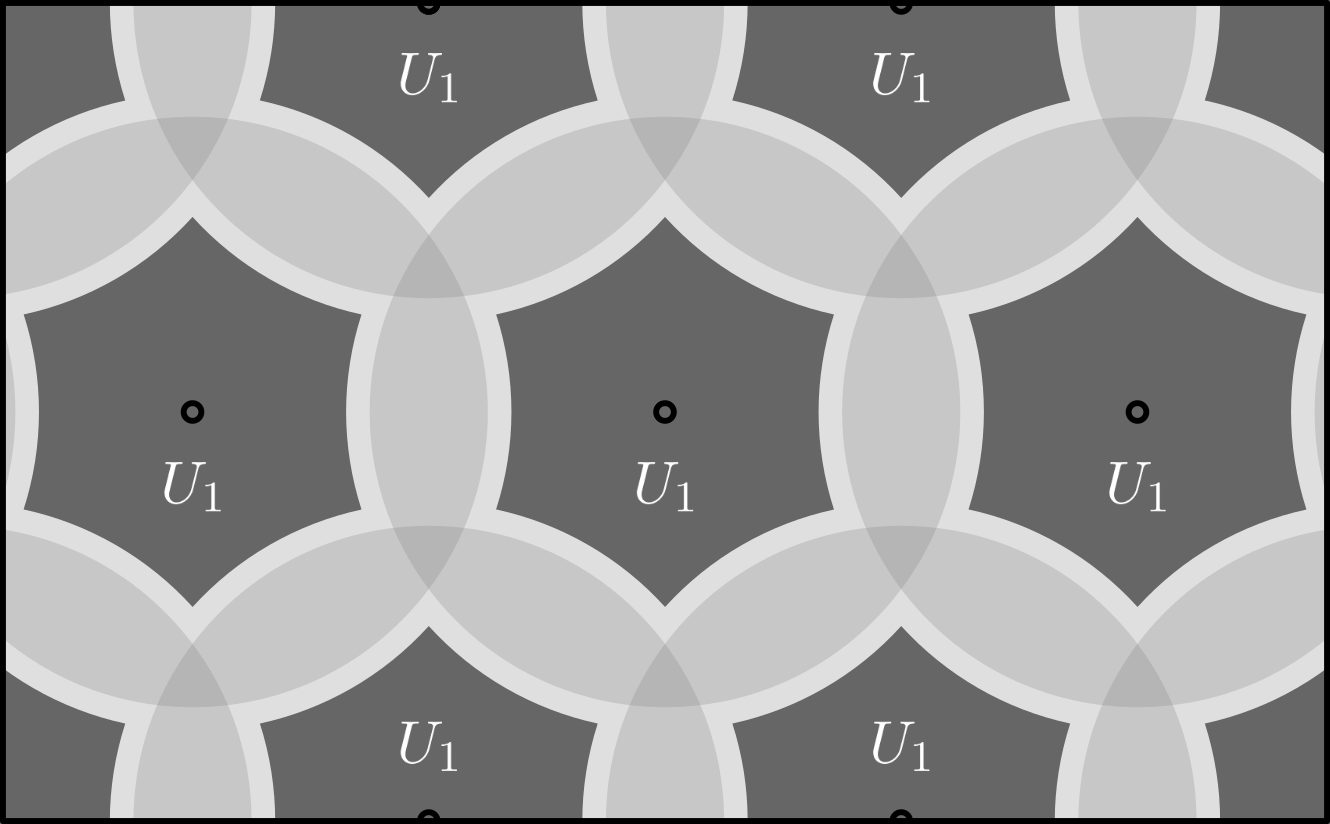}
    \caption{To construct $U_1$ (dark grey), for each center $c$ we let $T_c = $  the region of $G$ closer to $c$ than any other tile, then take the union over all $c$.}
    \label{fig:construction_2}
    \end{subfigure}
    \begin{subfigure}[t]{.4\textwidth}
    \centering
    \includegraphics[width=5cm]{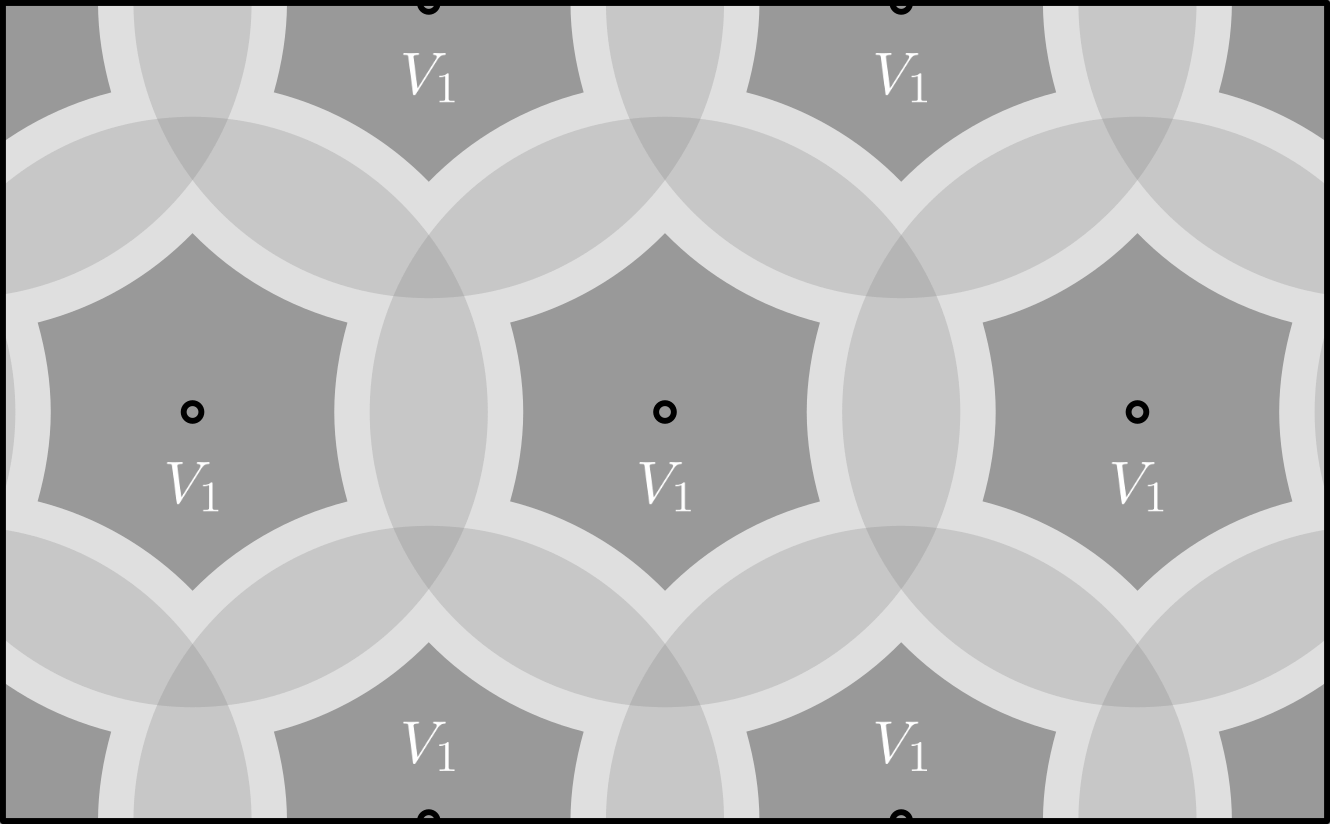}
    \caption{The set $V_1$ (medium grey) is found by taking $U_1$ and trimming off the boundary region.}
    \label{fig:construction_3}
    \end{subfigure}
    \hspace{1cm}
    \begin{subfigure}[t]{.4\textwidth}
    \centering
    \includegraphics[width=5cm]{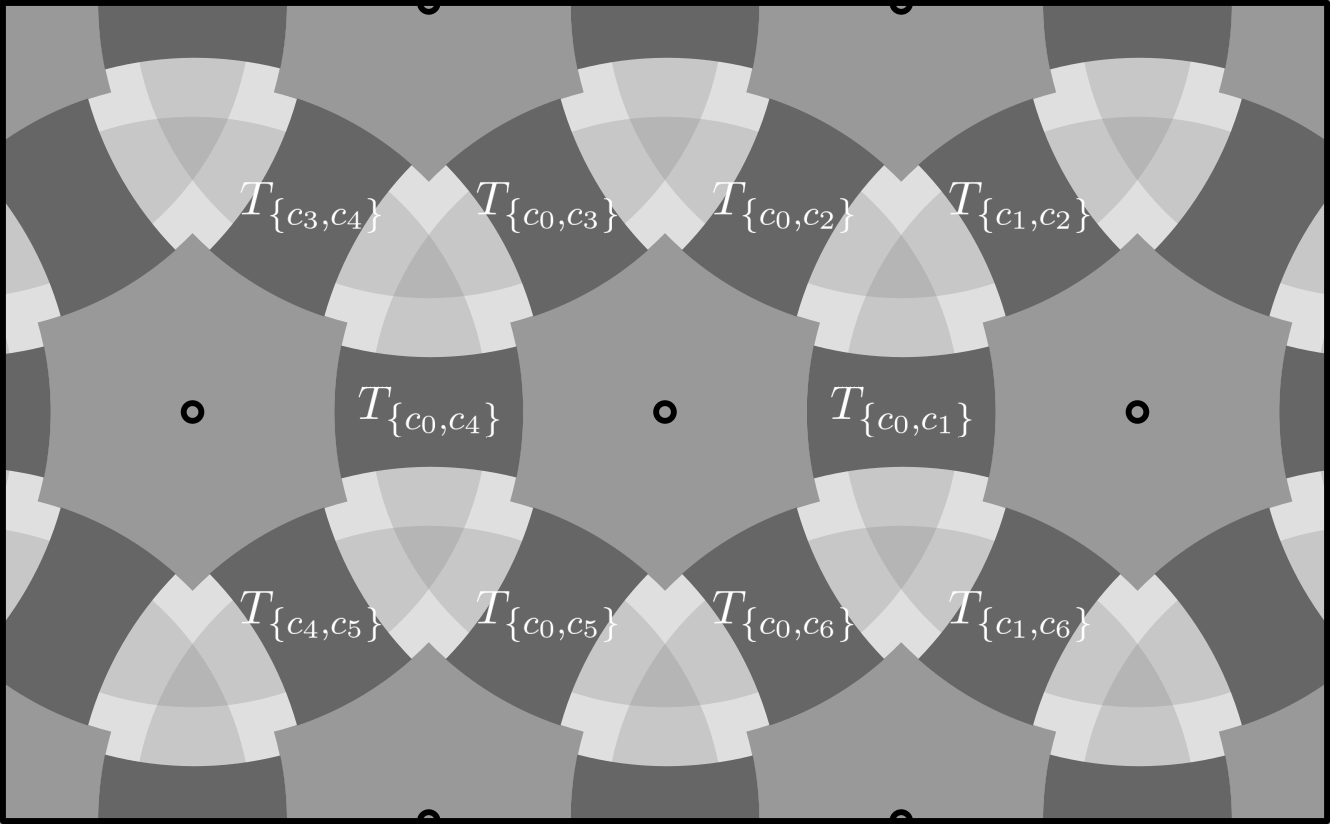}
    \caption{To construct $U_2$, for each pair of centers $c_i$, $c_j$ we let $T_{\{c_i,c_j\}} =$ the region of $G$ closest to both $c_i$ and $c_j$, then union together with $V_1$.}
    \label{fig:construction_4}
    \end{subfigure}
    \begin{subfigure}[t]{.4\textwidth}
    \centering
    \includegraphics[width=5cm]{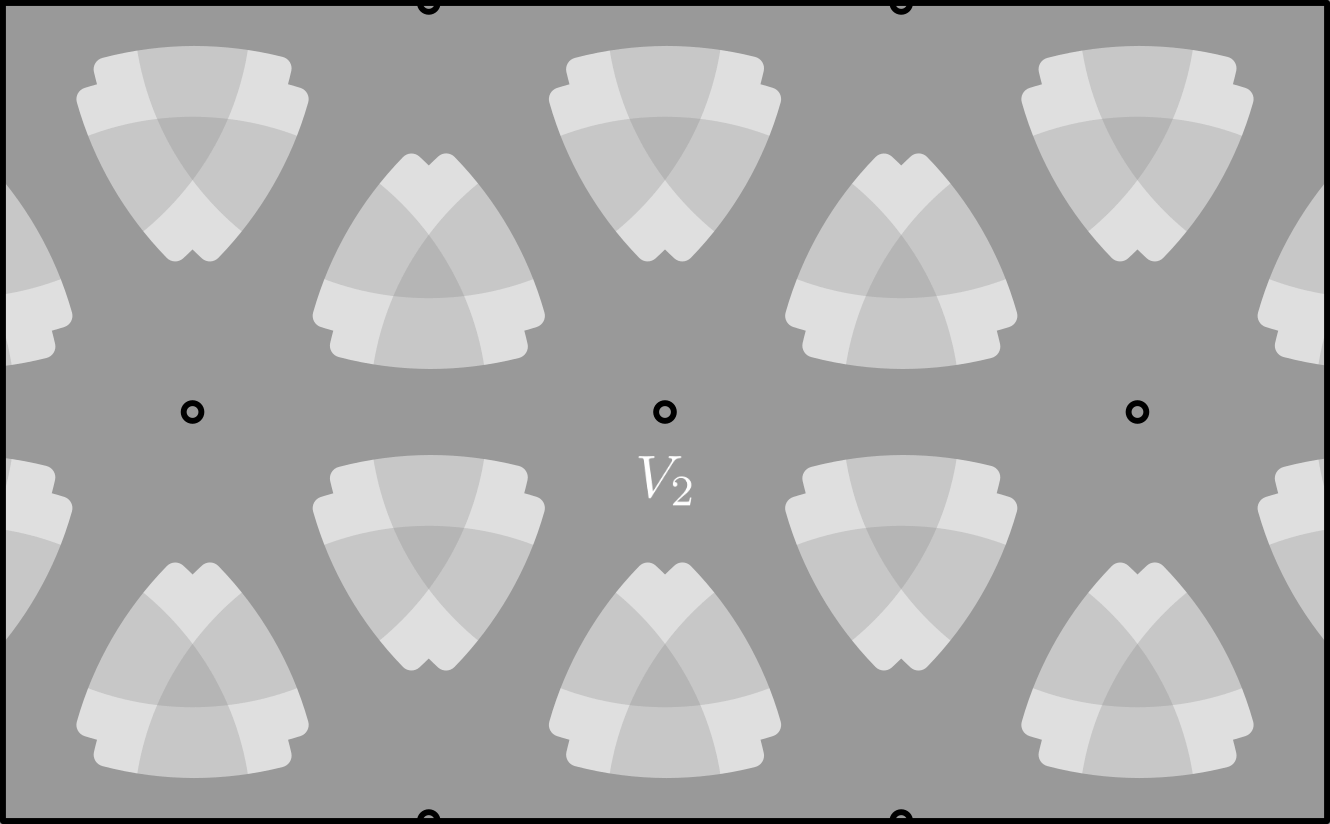}
    \caption{The set $V_2$ is obtained by trimming off the boundary of $U_2$.}
    \label{fig:construction_5}
    \end{subfigure}
    \hspace{1cm}
    \begin{subfigure}[t]{.4\textwidth}
    \centering
    \includegraphics[width=5cm]{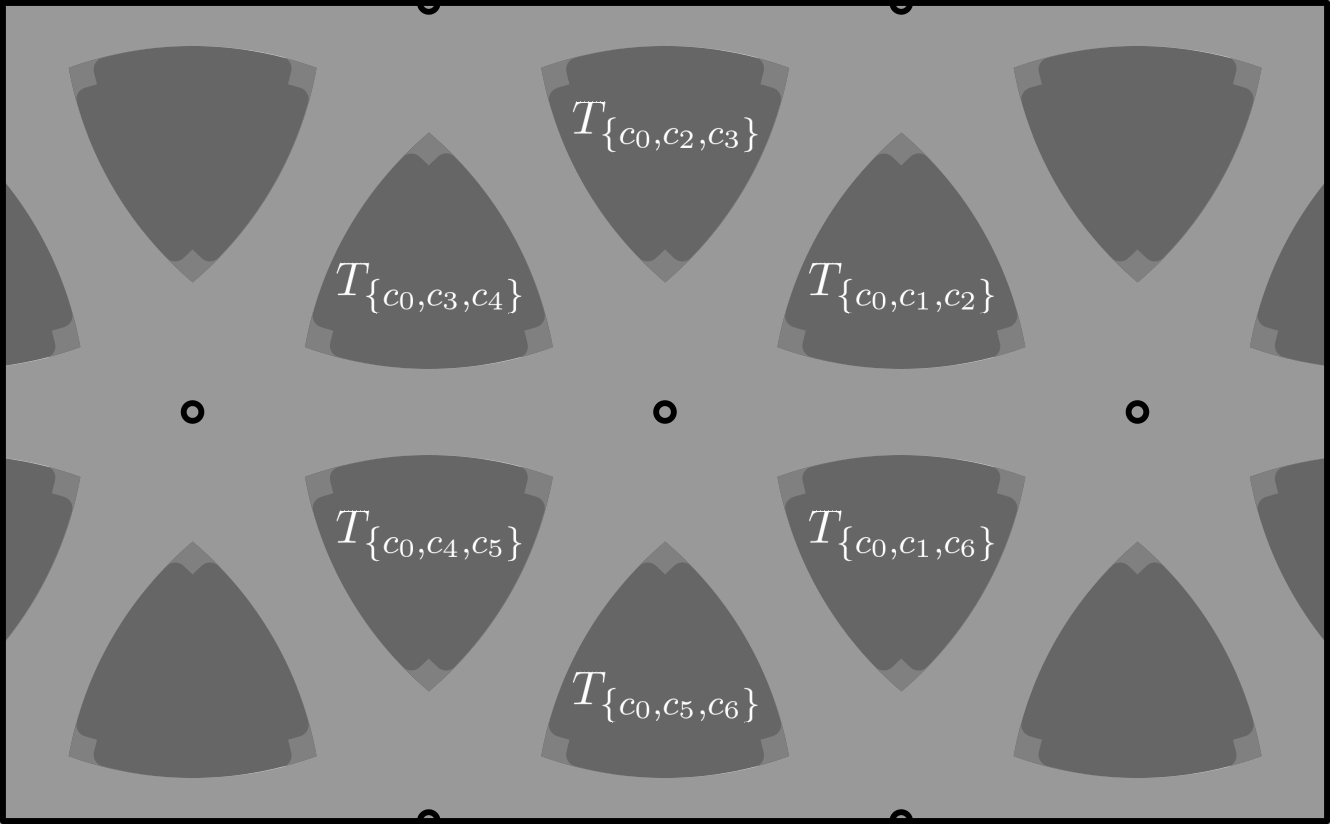}
    \caption{For $m = 3$ we consider sets of the form $T_{c_i,c_j,c_k}$, which are intersections of three expanded tiles. In this case, $U_3 = G$.}
    \label{fig:construction_6}
    \end{subfigure}
\caption{An illustration of the construction of $U_1$, $V_1$, $U_2$, $V_2$, and $U_3$ in a hypothetical case where $G = \mathbb{Z}^2$ is tiled by circles.}
\label{fig:construction}
\end{figure}

\begin{lemma} \label{Lemma:Disjoint}
For each $m \in [1,m_0]$, the collection $\{T_{S} K\}_{S \in \mathcal{C}_m}$ is pairwise disjoint.
\end{lemma}
\begin{proof}
Suppose $1 \leq m \leq m_0$ and we have $T_{S} K \cap T_{S'} K \neq \varnothing$ for some $S, S' \in \mathcal{C}_m$. Then there exists $g \in T_{S}$ such that $g \in T_{S'}K^2$. Since $g \in T_S \subset \tilde{T}_{S}$, we see that $\dist_K(g,cK^{2n_0}) \leq 3m-3$ for all $c \in S$, and thus $m \leq \deg_{3m-3}(g) \leq \deg_{3m-1}(g)$. Furthermore, since $g \in T_S \subset E_{m,3m-1}$, we have that $\deg_{3m-1}(g) \leq m$, and therefore $\deg_{3m-1}(g) = m$. By definition, this means that there are exactly $m$ elements $c \in \mathcal{C}$ such that $\dist(g,cK^{2n_0}) \leq 3m-1$. Now since $g \in T_{S'}K^2$, for each $c \in S'$, we get $\dist_K(g,cK^{2n_0}) \leq (3m-3)+2=3m-1$. Since $|S| = |S'| = m$ and $\dist(g,cK^{2n_0}) \leq 3m-1$ for all $c \in S \cup S'$, we must have that $S = S'$. 
\end{proof}

\begin{lemma} \label{Lemma:Contains}
For each $m \in [1,m_0]$, we have
\begin{equation} \label{Eqn:UmContains}
U_m \supset E_{m,3m-1},
\end{equation}
and
\begin{equation} \label{Eqn:VmContains}
V_m \supset E_{m,3m}.
\end{equation}
\end{lemma}
\begin{proof}
We proceed by induction on $m$. For the base case ($m=1$), since $\{cK^{2n_0}\}_{c \in \mathcal{C}}$ covers $G$, we get
\begin{equation*}
U_1 = \bigcup_{c \in \mathcal{C}} T_c = \bigcup_{c \in \mathcal{C}} cK^{2n_0} \cap E_{1,2} = \left( \bigcup_{c \in \mathcal{C}} cK^{2n_0} \right) \cap E_{1,2} = E_{1,2}.
\end{equation*}
Furthermore, for each $c \in \mathcal{C}$, we have 
\begin{equation*}
\int_K(cK^{2n_0} \cap E_{1,2}) \supset cK^{2n_0} \cap E_{1,3}.
\end{equation*}
Then we obtain that 
\begin{align*}
V_1 &= \int_K(U_1)\\
&\supset \bigcup_{c \in \mathcal{C}} \int_K(cK^{2n_0} \cap E_{1,2})\\
&\supset \bigcup_{c \in \mathcal{C}} cK^{2n_0} \cap E_{1,3}\\ 
&= \left( \bigcup_{c \in \mathcal{C}} cK^{2n_0} \right) \cap E_{1,3} \\
&= E_{1,3},
\end{align*}
where we have again used that $\{cK^{2n_0}\}_{c \in \mathcal{C}}$ covers $G$.

Now assume by induction that (\ref{Eqn:UmContains}) and (\ref{Eqn:VmContains}) hold for some $m < m_0$. Then
\begin{equation} \label{Eqn:InductionUm}
U_{m+1}  = \left( \bigcup_{S \in \mathcal{C}_{m+1}} T_{S} \right) \cup V_m 
 \supset  \left( \bigcup_{S \in \mathcal{C}_{m+1}} \tilde{T}_{S} \cap E_{m+1,3m+2}\right) \cup E_{m,3m},
\end{equation}
where we have used the inductive hypothesis (2).

Let $g \in E_{m+1,3m+2}$. If $\deg_{3m}(g) \leq m$, then $g \in E_{m,3m} \subset U_{m+1}$ by (\ref{Eqn:InductionUm}). Suppose instead that $\deg_{3m}(g) > m$. Then we have that $\deg_{3m}(g) = m+1$, since $\deg_{3m}(g) \leq \deg_{3m+2}(g) \leq m+1$. Then there exists $S \in \mathcal{C}_{m+1}$ such that 
\begin{equation*}
g \in \left( \bigcap_{c \in S} cK^{2n_0+ 3m} \right) \cap E_{m+1,3m+2} =  \tilde{T}_{S} \cap E_{m+1,3m+2} \subset U_{m+1},
\end{equation*}
where we have again used (\ref{Eqn:InductionUm}).
This verifies (\ref{Eqn:UmContains}) for $m+1$.

Finally, let $g \in E_{m+1,3m+3}$. Then $gK \subset E_{m+1,3m+2} \subset U_{m+1}$. Hence $g \in \int_K(U_{m+1})$, which establishes (\ref{Eqn:VmContains}) for $m+1$ and completes the induction.
\end{proof}

\begin{lemma} \label{Lemma:AllOfG}
$V_{m_0} = G$.
\end{lemma}
\begin{proof}
By Lemma \ref{Lemma:Contains} and our choice of $n_0$, we have $V_{m_0} \supset E_{m_0,3m_0} = E_{m_0,n_0}$. By Lemma \ref{Lemma:BoundedDegree}, for all $g \in G$, we have $\deg_{n_0}(g) \leq m_0$. Hence $G = E_{m_0,n_0} \subset V_{m_0}$. 
\end{proof}

\subsection{The homomorphism $\phi_1 : Z \to Y$}

In this section we describe how to use the FEP and the sets constructed in the previous section to define the desired homomorphism. To begin, we construct a type of block map to be used in the construction that follows.
Let $D$ be the set of all tuples $(A,F,v)$ such that $e \subset A  \subset K^{10n_0}$, $F \subset K^{10n_0}$, and $v \in \mathcal{L}_F(Y)$. Note that we allow tuples of the form $(A,\varnothing, \lambda)$, where $\lambda$ is the empty word.  We now construct a map that extends globally allowed patterns on $F$ to globally allowed patterns on $A \cup F$. More precisely, we claim that there is a map $\Phi : D \to \mathcal{L}(Y)$ satisfying the following conditions:
\begin{itemize}
\item[(i)] if $(A,F,v) \in D$, then $\Phi(A,F,v)$ is in $\mathcal{L}_A(Y)$;
\item[(ii)] the concatenation $\Phi(A,F,v) \cup v$ is well-defined and contained in $\mathcal{L}_{A \cup F}(Y)$;
\item[(iii)] if $(A,F,v)$ and $(gA,gF,g \cdot v)$ are both in $D$, then $\Phi(gA,gF,g\cdot v) = g \cdot \Phi(A,F,v)$.
\end{itemize}
 Let us construct such a map. First note that $K^{10n_0}$ is a finite set, and therefore $D$ is a finite set. Next, define an equivalence relation on $D$ as follows: for $(A_1,F_1,v_1)$ and $(A_2,F_2,v_2) \in D$, we set $(A_1,F_1,v_1) \sim (A_2,F_2,v_2)$ whenever there exists $g \in G$ such that $(A_1,F_1,v_1) = (gA_2,gF_2,g \cdot v_2)$. Note that this is indeed an equivalence relation. Let $\{C_1,\dots,C_{r_0}\}$ be the partition of $D$ into equivalence classes. For each $r \in \{1,\dots,r_0\}$, choose $(A_r,F_r,v_r)$ from $C_r$ arbitrarily. Since $v_r \in \mathcal{L}(Y)$, there exists a point $y_r \in Y$ such that $y_r(F_r) = v_r$. Define $\Phi(A_r,F_r,v_r) = y_r(A_r)$. Now let $(gA_r, gF_r, g \cdot v_r)$ be an arbitrary element of $C_r$. We define $\Phi(g A_r, gF_r, g \cdot v_r) = g \cdot \Phi(A_r,F_r,v_r)$. This defines $\Phi$ on all of $D$, and we note that properties (i) - (iii) hold by construction.

Now we define a map $\phi_1 : Z \to Y$. For each $z = \chi_{\mathcal{C}} \in Z$, the map is defined inductively in $m_0$ stages. At stage $m$, we first define a configuration on $U_m = U_m(z)$ that contains no forbidden patterns from $Y$ (i.e., it's locally allowed), and then we define a configuration on $V_m = V_m(z)$ that appears in a point in $Y$ (i.e., it's globally allowed). 

To begin the induction, let $m = 1$. Note that for each $c \in \mathcal{C}$, we have $e \in c^{-1} T_c \subset K^{2 n_0}$. Let $A_c = c^{-1} T_c$. Then $(A_c,\varnothing,\lambda) \in D$, and we let $p_c = c \cdot \Phi(A_c,\varnothing,\lambda) \in \mathcal{L}_{T_c}(Y)$. 
By the disjointness of $\{T_c K \}_{c \in \mathcal{C}}$ (Lemma \ref{Lemma:Disjoint}), the concatenation $u_1 = \sqcup p_c$ is well-defined and contains no forbidden patterns. Note that $u_1$ has domain $U_1$ and $V_1 = \int_K(U_1)$. Then by the FEP there is a point $y_1 \in Y$ such that $y_1(V_1) = u_1(V_1)$ (where we have used the equivalent characterization of the FEP in Remark \ref{Remark:FEP-Restatement}). Let $v_1 = u_1(V_1)$, and note that $v_1 \in \mathcal{L}_{V_1}(Y)$.

Now suppose by induction that for some $1 < m \leq m_0$ we have constructed a configuration $v_{m-1} \in \mathcal{L}_{V_{m-1}}(Y)$.  Consider an arbitrary nonempty $T_{S}$ with $S \in \mathcal{C}_m$.
Choose $g \in T_{S}$ arbitrarily.  (We argue below that the definition of the pattern $p_S$ does not depend on this choice.) Let $A_{S} = g^{-1}T_{S}$, and let
\begin{equation*}
F_{S} = \left( \bigcup_{c \in S} cK^{2n_0 + 3m-3} \right) \cap V_{m-1}.
\end{equation*}
Let $w_{S} = v_{m-1}(F_S)$. One may easily check that $(A_S,g^{-1}F_S, g^{-1}w_S) \in D$. Then let $p_S = g \cdot \Phi(A_S,g^{-1}F_S, g^{-1} \cdot w_S)$, and note that the concatenation $p_S \cup w_S$ is well-defined and contained in $\mathcal{L}_{T_S \cup F_S}(Y)$ by properties (i) and (ii) of $\Phi$. Furthermore, $p_S$ is independent of the choice of $g \in T_S$ by property (iii) of $\Phi$. Now observe that the concatenation
\begin{equation*}
u_m = \bigcup_{S \in \mathcal{C}_m} p_S \cup w_S = \left( \bigcup_{S \in \mathcal{C}_m} p_S \right) \cup v_{m-1}
\end{equation*}
is well-defined and contains no forbidden patterns (using Lemma \ref{Lemma:Disjoint}). Also, the shape of $u_m$ is $U_m$ and we have $V_m = \int_K(U_m)$. Then by the FEP, there exists $y_m \in Y$ such that $y_m(V_m) = u_m(V_m)$. Let $v_m = u_m(V_m)$, and note that $v_m \in \mathcal{L}(Y)$. This completes the induction. 

Finally, to define the map $\phi_1 : Z \to Y$ at the point $z$, we set $\phi_1(z) = v_{m_0}$, which is defined on all of $G$ by Lemma \ref{Lemma:AllOfG} and contained in $Y$ by construction. It remains to show that $\phi_1$ is continuous and shift-commuting.

\begin{lemma}
The map $\phi_1 : Z \to Y$ constructed in this manner is a homomorphism.
\end{lemma}
\begin{proof} 
First, observe that for all $n \in [0,n_0]$, we have $\dist_K(g,cK^{2n_0}) \leq n$ if and only if $c \in gK^{2n_0+n}$, and therefore the map $z \mapsto \deg^z_n$ can be written as a sliding block code with shape $K^{3n_0}$. Hence, for all $m \in [1,m_0]$ and $n \in [0,n_0]$, the map $z \mapsto \chi_{E^z_{m,n}}$ can be written as a sliding block code with shape $K^{3n_0}$. Next, since $\sigma^g z = \chi_{g \mathcal{C}}$, we have that $\mathcal{C}_m(\sigma^gz) = g \cdot \mathcal{C}_m(z)$. Furthermore, for all $n \in [0,n_0]$ and all $g \in G$, for any collection of $m \leq m_0$ group elements $g_1,\dots,g_m \in G$, we have
\begin{equation*}
g \in \bigcap_{j=1}^m g_j K^{2n_0 + n} \quad \iff \quad \forall j=1,\dots,m, \, g_j \in gK^{2n_0+n}.
\end{equation*}
Since the latter condition is determined by $gK^{3n_0}$, we conclude by induction on $m$ that the maps $z \mapsto \chi_{U_m(z)}$ and $z \mapsto \chi_{V_m(z)}$ can be written as block codes with shape $K^{3n_0}$. 
Finally, we note that at each stage of the induction, any symbol defined in $u_m$ and $v_m$ at $g \in G$ can be determined in a shift-commuting manner by the {pattern} $z(g K^{10n_0})$ (depending on the patterns $u_{m-1}$ and $v_{m-1}$, the sets $U_m$ and $V_m$, and the map $\Phi$). Since there are only $m_0$ stages of the induction (for all $z \in Z$), we obtain that the map $\phi_1$ is a sliding block code and hence a homomorphism.
\end{proof}

\section{Proof of Corollary \ref{Cor:Embedding}} \label{Sect:Cor}

Finally, we present a short proof of Corollary~\ref{Cor:Embedding} by combining Theorems \ref{Thm:bland} and \ref{Theorem:Main-Result}. We need only to quickly check that the hypotheses of Theorem~\ref{Thm:bland} are implied by the hypotheses and the conclusion of Theorem~\ref{Theorem:Main-Result}. For convenience, we provide a restatement of Corollary~\ref{Cor:Embedding}.

\begin{embedding-corollary}
Let $G$ be a countable group such that every finitely generated subgroup of $G$ has polynomial growth. 
Suppose $X$ is a nonempty aperiodic $G$-subshift and $Y$ is a $G$-subshift with the finite extension property and no global period. If $h(X) < h(Y)$, then $X$ embeds into $Y$.
\end{embedding-corollary}

\begin{proof}
    We claim first that the hypotheses imply that $G$ is amenable. First, note that (1) a group is amenable if and only if all of its finitely generated subgroups are amenable, since amenability is preserved under taking subgroups and directed unions, and (2) any finitely generated group of subexponential growth is  amenable. Thus, $G$ is amenable since every finitely generated subgroup of $G$ is of subexponential growth. Furthermore, since every finitely generated subgroup of $G$ has subexponential growth, $G$ also satisfies the comparison property \cite[Theorem~6.33]{downarowicz_zhang-symb-ext}. 

    Next we note that since $Y$ satisfies the FEP, it follows that $Y$ is a strongly irreducible SFT (Proposition~\ref{Prop:SIandSFT}). Given that $X$ is aperiodic,  Theorem~\ref{Theorem:Main-Result} gives a homomorphism $\phi : X \to Y$. Given also that $h(X) < h(Y)$ {and $Y$ has no global period}, we see that every hypothesis of Theorem~\ref{Thm:bland} is satisfied. Thus, there exists an embedding of $X$ into $Y$.
\end{proof}

\section*{Acknowledgments}
The authors thank Ronnie Pavlov for productive conversations and Mike Boyle for many helpful comments. KM acknowledges the support of the National Science Foundation through the grants DMS-1847144 and DMS-2113676.

\bibliographystyle{plain}
\bibliography{references}

\end{document}